\documentclass[oneside,11pt]{amsart}
\usepackage{amsmath, amsfonts,amsthm,times,graphics}

\usepackage[active]{srcltx}
 \makeatletter
\renewcommand*\subjclass[2][2000]{%
  \def\@subjclass{#2}%
  \@ifundefined{subjclassname@#1}{%
    \ClassWarning{\@classname}{Unknown edition (#1) of Mathematics
      Subject Classification; using '1991'.}%
  }{%
    \@xp\let\@xp\subjclassname\csname subjclassname@#1\endcsname
  }%
}
 \makeatother
\newtheorem{theorem}{Theorem}[section]
\newtheorem{lemma}[theorem]{Lemma}

\newtheorem{proposition}[theorem]{Proposition}
\theoremstyle{definition}

\newtheorem{example}[theorem]{Example}
\newtheorem{remark}[theorem]{Remark}
\numberwithin{equation}{section}
\newcommand{\abs}[1]{\lvert#1\rvert}

\begin{document}

\title{On quasiconformal harmonic maps between surfaces}



\author{David Kalaj}
\address{University of Montenegro, Faculty of Natural Sciences and
Mathematics, Cetinjski put b.b. 81000 Podgorica, Montenegro}
\email{davidk@t-com.me} \subjclass {Primary 30C55, Secondary 31C05}
\begin{abstract}
The following theorem is proved: If $w$ is a quasiconformal harmonic
mapping between two Riemann surfaces with compact and smooth
boundaries and approximate analytic metrics, then $w$ is
bi-Lipschitz continuous with respect to internal metrics. If the
surfaces are subsets of the Euclidean spaces, then $w$ is
bi-Lipschitz with respect to the Euclidean metrics.
\end{abstract} \maketitle

\section{Introduction}

\subsection{The main definitions and notation}

By $\Bbb U$ is denoted the unit disk, and by $S^1$ is denoted its
boundary. By $D$ and $\Omega$ are denoted domains in complex plane
$\mathbb C$.

Let $(\Sigma_1,\sigma)$ and $(\Sigma_2,\rho)$ be Riemann surfaces
 (with or without boundary),
with metrics $\sigma$ and $\rho$ respectively. We say that a mapping
$w$ between Riemann surfaces $(\Sigma_1,\sigma)$ and
$(\Sigma_2,\rho)$ is bi-Lipschitz, if there exist constants $q>0$
and $Q>0$, such that
$$q d_{\sigma}(z_1,z_2)\le d_{\rho}(w(z_1),w(z_2))\le
Qd_{\sigma}(z_1,z_2), z_1,z_2\in \Sigma_1,$$ where
$$d_{\varrho}(z_1,z_2) = \inf_{\gamma\in \Gamma_{z_1,z_2}}\int_\gamma \varrho(z)
|dz|\quad  \varrho\in\{\sigma,\rho\},$$ and
$$\Gamma_{z_1,z_2} = \{\gamma : \gamma \text{ is a rectifiable curve joining $z_1$ and $z_2$ in $\Sigma_1$}\}.$$

If $f:(\Sigma_1,\sigma)\to(\Sigma_2,\rho)$ is a $C^2$ mapping, then
$f$ is said to be harmonic with respect to $\rho$ (abbreviated
$\rho$-harmonic) if

\begin{equation}\label{el}
f_{z\overline z}+{(\log \rho^2)}_w\circ f f_z\,f_{\bar z}=0,
\end{equation}
where $z$ and $w$ are the local parameters on $\Sigma_1$ and
$\Sigma_2$ respectively. From \eqref{el} we see that the harmonicity
of $f$ depends only on the conformal structure, but not on the
particular metric of $\Sigma_1$.

Also $f$ satisfies \eqref{el} if and only if its H. Hopf
differential
\begin{equation}\label{anal}
\Psi=\rho^2 \circ f f_z\overline{f_{\bar z}}
\end{equation} is a
holomorphic quadratic differential on $\Sigma_1$.

For $g:\Sigma_1 \mapsto \Sigma_2$ the energy integral is defined by

\begin{equation} E[g,\rho]=\int_{\Sigma_1}\rho^2\circ g
(|\partial g|^2+|\bar \partial g|^2) dV_\sigma,
\end{equation}
where $\partial g$, and $\bar \partial g$ are the partial
derivatives taken with respect to the metrics $\varrho$ and
$\sigma$, and $dV_\sigma$ is the volume element on
$(\Sigma_1,\sigma)$. Assume that energy integral of $f$ is bounded.
Then if $f$ is a critical point of the corresponding functional,
where the homotopy class of $f$ is the range of this functional,
then $f$ is harmonic.

We will consider harmonic mappings between compact Riemann surfaces
with boundaries, with respect to a metric $\rho$, where the metric
$\rho$ satisfies the following inequality
$$|{(\log \rho^2)}_w| = \frac{|\nabla\rho|}{\rho}\le M,$$ where $M$
is a constant (with respect to local parameters). Under this
condition, if for example the domain of $\rho$ is the unit disk,
then there hold the double inequality
\begin{equation}\label{double}\rho(0)e^{-M}\le {\rho(w)}\le \rho(0) e^{M}.\end{equation}  Such
metrics are called approximately analytic \cite{EH}. The spherical
metric
$$\rho(w) = \frac{2}{1+|w|^2}$$ is approximately analytic, but the
hyperbolic metric \begin{equation}\label{hiperbola}\lambda(w) =
\frac{2}{1-|w|^2}\end{equation} is not. Let us mention the following
important fact. Equation \eqref{el} is equivalent to the following
system of equations, which can be directly extended to the
dimensions bigger than 2:

\begin{equation}\label{cron} \Delta u^i + \sum_{\alpha, \beta, k, \ell = 1}^2\Gamma^i_{k\ell}(u)D_\alpha u^k D_\beta u^\ell=0, \,\, i =
1,2\quad (f=(u^1,u^2))\end{equation}

where $\Gamma^i {}_{k\ell}$ are Christoffel Symbols of the metric
$\rho$ (or of a metric tensor $(g_{jk}\ )$):
$$    \Gamma^i {}_{k\ell}=\frac{1}{2}g^{im} \left(\frac{\partial
g_{mk}}{\partial x^\ell} +
    \frac{\partial g_{m\ell}}{\partial x^k} - \frac{\partial g_{k\ell}}{\partial x^m} \right) =
     {1 \over 2} g^{im} (g_{mk,\ell} + g_{m\ell,k} - g_{k\ell,m}),$$
and the matrix $(g^{jk}\ )$ is an inverse of the metric tensor
$(g_{jk}\ )$.

It can be easily seen that since \eqref{cron} and \eqref{el} are
equivalent, a metric $\rho$ is approximate analytic if and only if
Christoffel symbols are bounded.

Let $P$ be the Poisson kernel, i.e. the function

$$P(z,e^{i\theta})=\frac{1-|z|^2}{|z-e^{i\theta}|^2},$$ and let
$G$ be the Green function of the unit disk with respect to the
Laplace operator, i.e. the function
\begin{equation}\label{green1}
G(z,w)=\frac{1}{2\pi}\log\left|\frac{1-z\overline w}{z-w}\right|\,
\, z,w\in \Bbb U,\ z\neq w.
\end{equation}

Let $f:S^1\to \mathbb C$ be a bounded integrable function on the
unit circle $S^1$ and let $g:U\to \Bbb C$ be continuous. The
solution of the equation $\Delta w=g$ (in weak sense) in the unit
disk satisfying the boundary condition $w|_{S^1}=f\in L^1(S^1)$ is
given by
\begin{equation}\label{e:POISSON1}\begin{aligned}w(z)&=P[f](z)- G[g](z)\\&:
=\frac{1}{2\pi}\int_0^{2\pi}P(z,e^{i\varphi})f(e^{i\varphi})d\varphi-
\int_{\mathbb{U}}G(z,\omega)g(\omega)\,dm(\omega),\,
\end{aligned}\end{equation} $|z|<1,$ where $dm(\omega)$ denotes the
Lebesgue measure in the plane. It is well known that if $f$ and $g$
are continuous in $S^1$ and in $\overline{\mathbb{U}}$ respectively,
then the mapping $w=P[f]-G[g]$ has a continuous extension $\tilde w$
to the boundary, and $\tilde w = f$ on $S^1$. See \cite[pp.
118--120]{Her}.

Let  $0\le k<1$ and let $K = \frac{1+k}{1-k}$. An orientation
preserving diffeomorphism $w$ between two Riemann surfaces is called
 $K-$ quasiconformal (abbreviated q.c.) if $${|w_{\bar
z}|} \le k{|w_z|}\ \ \  \text{(in local coordinates)}.$$ The
previous inequality can be written us $$|\nabla w(z)|\le K l(\nabla
(w(z))),$$ where $$|\nabla w(z)|:=\sup\{|\nabla w(z)h|: |h| =1\} =
|w_z|+|w_{\bar z}|$$ and

$$l(\nabla w(z)):=\inf\{|\nabla w(z)h|: |h| =1\}=|w_z|-|w_{\bar z}|.$$ See \cite{Ahl} and
\cite{sl} for the definition of arbitrary quasiconformal mapping
between plane domains, Euclidean surfaces or Riemann surfaces.

\subsection{Background}


In this paper we deal with q.c. $\rho$ harmonic mappings and study
their global bi-Lipschitz character. See \cite{Om} for the
pioneering work on this topic and see \cite{hs} for related earlier
results. In some recent papers, a lot of work has been done on this
class of mappings (\cite{MMM}-\cite{kalajan} , \cite{MP}, \cite{pk},
\cite{kojic}). In these papers it is established the bi-Lipschitz
character of q.c. harmonic mappings between plane domains with
certain boundary conditions. The most important results of these
papers is that Euclidean harmonic q.c. mappings between plane
domains with smooth boundaries are Euclidean quasi-isometries (and
consequently hyperbolic quasi-isometries). Notice that in general,
quasi-symmetric selfmappings of the unit circle do not provide
quasiconformal Euclidean harmonic extension to the unit disk. The
case of hyperbolic harmonic mapping is an open attractive problem,
known as Schoen conjecture \cite{sy}. More precisely Schoen
conjectured that, every quasi-symmetric selfmappings of the unit
circle has q.c. hyperbolic harmonic extension. Further in\cite{Om}
is given an example of $C^1$ diffeomorphism of the unit circle onto
itself, whose Euclid harmonic extension is not Lipschitz.

In contrast to the case of Euclidean metric, in the case of the
hyperbolic metric, if $f:S^1\mapsto S^1$ is a $C^1$ diffeomorphism,
or more general if $f:S^{n-1}\mapsto S^{m-1}$ is a mapping with a
non-vanishing energy, then its hyperbolic harmonic extension is
$C^1$ up to the boundary (\cite{lit}) and (\cite{lit1}). On the
other hand Wan (\cite{wan}) showed that q.c. hyperbolic harmonic
mappings between smooth domains are hyperbolic quasi-isometries (but
in general they are not Euclidean quasi-isometries, neither its
boundary values is absolutely continuous, in general). See also
\cite{litw} for the generalization of the last result to hyperbolic
Hadamard surfaces.

The starting position of this paper are the following recent
results.
\begin{proposition}\label{arbitrarydomain}\cite{kalmat}
Let $w$ be a quasiconformal $C^{2}$ homeomorphism from a bounded
plane domain $D$ with $C^{1,\alpha}$ boundary onto a bounded plane
domain $\Omega$ with $C^{2,\alpha}$ boundary. If there exist
constants $B$ and $C$ such that
\begin{equation}\label{apro0}
|\Delta w|\le B|\nabla w|^2+C\,, \quad  z\in D.\end{equation} then
$w$ has bounded partial derivatives in $D$, in particular $|\nabla w|_\infty<\infty$. 
\end{proposition}
\begin{proposition}\cite{pisa}
Let $w=f(z)$ be a $K-$quasiconformal Euclidean harmonic mapping
between a
 Jordan domain $D$ with $C^{1,\alpha}$ boundary and a
Jordan domain $\Omega$ with $C^{2}$ boundary. Then $w$ is bi-Lipschitz. 
\end{proposition}
Using a different approach, we extend the last result to the class
of harmonic mappings with respect  to approximate analytic metrics
and harmonic mappings between Riemann surfaces and more generally to
the class of mappings satisfying certain growth condition on
Laplacian.
\subsection{Main results}
The following three theorems will be proved in this paper.
\begin{theorem}[The main theorem]\label{tt}
If $w$ is a $C^{2}$ $K-$quasiconformal mapping of the unit disk onto
itself, satisfying the inequality
\begin{equation}\label{apro}|\Delta w|\le B|\nabla
w|^2, \ \ (z\in \Bbb U),\end{equation} then $w$ is bi-Lipschitz.
\end{theorem}

\begin{theorem}\label{coro} Assume that $\rho$ is an approximate analytic metric and let $w$ be a
$\rho$ harmonic q.c. selfmapping of the unit disk. Let $z_n$ be any
sequence of points of the unit disk and let in addition $p_n$ and
$q_n$ be M\"obius transformations such that $p_n(w(z_n))=0$ and
$q_n(0)=z_n$.

Then there exists a subsequence of $w_n = p_n\circ w\circ q_n$
converging to a $\rho_0$ harmonic mapping $w_0$, where $\rho_0$ is a
metric in the unit disk.
\end{theorem}

\begin{theorem}\label{tito} Let $(\Sigma_1,\sigma)$ and $(\Sigma_2,\rho)$ be
Riemann surfaces with smooth compact boundaries, with approximate
analytic metrics $\rho$ and $\sigma$. If $w: \Sigma_1\to \Sigma_2$
is a q.c. harmonic mapping, then $w$ is bi-Lipschitz.
\end{theorem}

\begin{remark}\label{nota} Notice that in Theorem~\ref{tt} $w$ is not assumed  to be a diffeomorphism. However a
Berg result \cite[Lemma~1]{berg} implies an extension of Lewy
theorem \cite{l} for a quasiconformal mapping $w$. More precisely,
the Berg result asserts that:
\\
Assume that the mapping $w:r\mathbb U\to \mathbb C$, $0<r\le 1$, is
one to one and satisfies the differential inequality $|\Delta w|\le
C|\nabla w|.$ Then $|\nabla w|>0$ on $r\mathbb U$, where $\mathbb U$
is the unit disk. By using the quasiconformality of $w$, we infer
that $J_w(z)\ge |\nabla w|^2/K>0$ for $|z|<r<1$. Thus $w$ is a
diffeomorphism. On the other hand every harmonic homeomorphism
between Riemann surfaces, under some curvature restrictions on the
metric of the target, is a diffeomorphism. For the previous result
and related results we refer to \cite{jj, dur, an, ly}. Thus the
homeomorphisms of Theorem~\ref{coro} and Theorem~\ref{tito} are
diffeomorphisms.
\end{remark}
Together with this introduction the paper contains two other
sections. The proof of main theorem (Theorem~\ref{tt}) is given in
the second section.
 Let us briefly explain the idea of the proof. Since by Proposition~\ref{arbitrarydomain}, a
 mapping $w$ satisfying the condition of Theorem~\ref{tt} is Lipschitz, all we need to show is the fact that it is co-Lipschitz,
i.e. its inverse mapping is Lipschitz. In order to show that the
mapping $w$ is co-Lipschitz, we will argue by contradiction,
 which means that there exist a sequence of the points $z_n$ from the unit disk such that
$\lim_{n\to \infty}\nabla w(z_n) = 0$.  In order to do so,
previously, we prove a version of Schwarz lemma for harmonic q.c.
mappings (Lemma~\ref{lemjun}). Then we take $w_n= p_n\circ w\circ
q_n$, where $p_n$ and $q_n$ are M\"obius transformations of the unit
disk onto itself such that $p(w(z_n))=0$ and $q(0)=z_n$. $w_n$ is a
sequence of q.c. harmonic mappings with respect to certain metrics
$\rho_n$ satisfying the normalization condition $w_n(0)=0$. This
sequence converges, up to some subsequence, to a q.c. harmonic
mapping $w_0$ with respect to a metric $\rho_0$. In proving the last
fact we will make use of Arzela-Ascoli theorem, Vitali theorem and
of representation formula \eqref{e:POISSON1}. To do so, we will
prove more, we will show that the sequence $w_n$, together with its
gradient $\nabla w_n$ converges to $w_0$ and $\nabla w_0$
respectively, uniformly in compact subsets of the unit disk. Several
time we will make use of Lemma~\ref{lemjun} and
Proposition~\ref{arbitrarydomain}. The fact that $w_0$ is q.c.
having a critical point and satisfying certain conditions, will
contradict Carleman-Hartman-Wintner lemma. The sequence $w_n$
converges, up to some subsequences, to some q.c. harmonic mapping
$w_0$ independently on the condition $\lim_{n\to \infty}\nabla
w(z_n) = 0$. This procedure, together with the Montel's theorem for
the Hopf differentials of the sequence $w_n$, will produce a new
metric $\rho_0$ and a $\rho_0$--q.c. harmonic mapping $w_0$. This
yields the proof of Theorem~\ref{coro}. Together with this proof,
the last section contains the proof of Theorem~\ref{tito}.

In the end of the paper it is shown that this method works for
hyperbolic metrics as well (which are not approximate analytic).
Hyperbolic metric and Euclidean metrics are not bi-Lipschitz
equivalent,  and therefore q.c. hyperbolic harmonic mappings are
not, in general, bi-Lipschitz mappings w.r. to Euclidean metric.

The conclusion is that every q.c. harmonic mappings between two
Riemann surfaces is bi-Lipschitz with respect to their corresponding
metrics. It remains an open problem if the quasi-conformality is
important in some results we prove.

In the following example it is shown that the inequality
\eqref{apro} in the main theorem cannot be replaced by the weaker
one \eqref{apro0}.

\begin{example}\cite{trans}
Let $w(z) = |z|^\alpha z$, with $\alpha>1$. Then $w$ is a twice
differentiable $(1+\alpha)-$quasiconformal self-mapping of the unit
disk. Moreover $$\Delta w = \alpha(2+\alpha)\frac{|z|^{\alpha}}{\bar
z}= g.$$ Thus $g=\Delta w$ is continuous and bounded by
$\alpha(2+\alpha)$. However $w^{-1}$ is not Lipschitz, because
$l(\nabla w)(0) = |w_z(0)|-|w_{\bar z}(0)|=0$.
\end{example}

\section{The proof of main theorem}

The following important lemma lies behind our main results.
\begin{proposition}[The Carleman-Hartman-Wintner lemma]\cite{har} Let $\varphi\in C^1(D)$ be a
real-valued function satisfying the differential inequality
$$|\Delta \varphi|\le C(|\nabla \varphi|+|\varphi|)$$ i.e., $$\Delta \varphi = -W,\, \ \ \ \ \ |W(z)|\le C(|\nabla \varphi|+|\varphi|)$$

$(z\in D)$, in the weak sense. Suppose that $D$ contains the origin.
Assume that $\varphi(z) = o(|z|^n)$ as $|z|\to 0 $ for some $n \in
N_0$. Then
$$\lim_{z\to 0}\frac {\varphi_z(z)}{z^n}$$ exists.
\end{proposition}
The following proposition is a consequence of
Carleman-Hartman-Wintner lemma.
\begin{proposition}\label{propi}\cite[Proposition~7.4.3.]{sf1} Let $\{u_k(z) \}$ be a sequence
of real functions of class $C^1(D)$ satisfying the differential
inequality
\begin{equation}\label{(7.28)}|\Delta u_k|\le C(|\nabla u_k|+|u_k|)\end{equation}
where $C$ is independent of $k$. Assume that
\begin{equation}\label{(7.29)} u_k(z)\to u_0(z),\ \   \nabla u_k(z)
\to \nabla  u_0(z),\end{equation} uniformly in $D$ ($k\to \infty$).
Assume in addition
\begin{equation}\label{(7.30)}u_0(z) = o(|z|)\text{ as $|z|\to 0$, }\end{equation} and that \begin{equation}\label{(7.31)}\nabla u_k(z)\neq 0\text{ for all $k$ and $z\in D$}.\end{equation} Then
$u_0(z)\equiv 0$.
\end{proposition}
The proof of the main theorem is based on the following three
lemmas.
\begin{lemma}\label{lemjun}
If $w:\Bbb U\to \Bbb U$, $w(0)=0$, satisfies the conditions of
Theorem~\ref{tt}, then there exists a constant $C(K)$ such that
\begin{equation}\label{ccoo}\frac{1-|z|^2}{1-|w(z)|^2}\le C(K)\quad z\in \Bbb U.\end{equation}
\end{lemma}
\begin{proof}
Take
$$\mathcal{QC}(\Bbb U, B,K)=\{w:\Bbb U\to \Bbb U: w(0)=0, |\Delta w|\le B|\nabla w|^2, w\text{ is } K. q.c.\}.$$
Let us choose  $A$ such that the function $\varphi_w$, $w\in
{\mathcal{QC}(\Bbb U, B,K)}\}$ defined by
$$\varphi_w(z) = -\frac{1}{A}+\frac{1}{A}
e^{A(|w(z)|-1)}$$ is subharmonic in $\varrho:={4^{-K}}\le |z|\le 1$.

Take $$s = \frac{w}{|w|}, \ \ \  \rho = |w|.$$ As $w= s\rho$ is a
$K$ quasiconformal selfmapping of the unit disk with $w(0)=0$, by
Mori's theorem (\cite{wang}) it satisfies the double inequality:
\begin{equation}\label{mori}
\left|\frac z{4^{1-1/K}}\right|^K \le \rho\le 4^{1-1/K} |z|^{1/K}.
\end{equation}
By \eqref{mori} for $\varrho\le |z|\le 1$ where
\begin{equation}\label{varh}\varrho:={4^{-K}}\end{equation} we have
\begin{equation}\label{rhoeq}
\rho\ge \rho_0:=4^{1-K^2-K}.
\end{equation}
Now we choose $A$ such that
$$\frac{A\rho_0^2}{K^2}+2-2BK^2\ge
0.$$ Take
$$\chi(\rho) = -\frac{1}{A}+\frac{1}{A} e^{A(\rho-1)}.$$  Then $$\chi'(\rho) =  e^{A(\rho-1)}$$ and $$\chi''(\rho) = {A} e^{A(\rho-1)}.$$ On the other hand
\begin{equation}\label{kompi}\Delta \varphi_w(z) =\chi''(\rho)|\nabla |w||^2+\chi'(\rho)\Delta
|w|.\end{equation} Furthermore
\begin{equation}\label{delc}\Delta |w|=2|\nabla s|^2+
2\left<\Delta w,s\right> .\end{equation}
\\
To continue observe that \begin{equation}\label{observe}\nabla
w=\rho\nabla s + \nabla \rho \otimes s.\end{equation} Since
$$|\nabla w(z)|\le K l(\nabla
w(z)),$$ choosing appropriate unit vector $h$ we obtain the
inequality

\begin{equation}\label{22e}|\nabla w|\le K\rho|\nabla
s|.\end{equation} Similarly it can be proved the inequality
\begin{equation}\label{11e}K|\nabla |w||\ge \rho|\nabla
s|.\end{equation} Using \eqref{rhoeq}, \eqref{kompi}, \eqref{delc},
\eqref{apro}, \eqref{22e} and \eqref{11e}, it follows finally that
$$\Delta \varphi_w(z)\ge (\frac{A\rho_0^2}{K^2}+2-2BK^2)e^{A(\rho-1)} \ge 0,\ \ \  4^{-K} \le |z|\le 1.$$
Define
$$\gamma(z) = \sup \{\varphi_w(z): w\in {\mathcal{QC}(\Bbb U, B,K)}\}.$$
Prove that $\gamma$ is subharmonic for $4^{-K} \le |z|\le 1$. In
order to do so, we will first prove that $\gamma$ is continuous. For
$z,z'\in \Bbb U$ and $w\in {\mathcal{QC}(\Bbb U, B,K)}$, according
to Mori's theorem (see e.g. \cite{Ahl}), we have
\[\begin{split}\abs{\varphi_w(z) -\varphi_w(z')} & = \frac{1}{A}
\abs{(e^{A(|w(z)|-1)}-e^{A(|w(z')|-1)})}\\&\le |w(z) - w(z')|\le
16|z-z'|^{1/K}.\end{split}
\]
Therefore $$|\gamma(z)- \gamma(z')|\le 16|z-z'|^{1/K}.$$ This means
in particular that $\gamma$ is continuous. It follows that $\gamma$
is subharmonic since it is the supremum of subharmonic functions
(see e.g. \cite[Theorem~1.6.2]{hol}).

If $|z_1|=|z_2|$ then $\gamma(z_1)=\gamma(z_2)$. In order to prove
the last statement we do as follows. For every $\varepsilon
>0$ there exists some $w\in {\mathcal{QC}(\Bbb U, B,K)}$
such that $$\varphi_w(z_2)\le \gamma(z_2) \le
\varphi_w(z_2)+\varepsilon.$$ Now  $w_1(z)=w(\frac{z_2}{z_1} z)$ is
in the class ${\mathcal{QC}(\Bbb U, B,K)}$. Therefore
$$\varphi_{w_1}(z_1)\le \gamma(z_1) \le \varphi_{w_1}(z_1)+\varepsilon.$$
As $\varepsilon$ is arbitrary and as $w_1(z_1) = w(z_2)$ it follows
that $\gamma(z_1)=\gamma(z_2)$.

This yields that $$\gamma(z) =
g(r)=-\frac{1}{A}+\frac{1}{A}e^{A(h(r)-1)}.$$
It is well known that a radial subharmonic function is an increasing
convex function of $t = \log r$ , for $-\infty < t < 0$ (see
\cite[Theorem 2.6.6]{rt}). From \eqref{mori} $$g(4^{-K}) \le
-\frac{1}{A} + \frac{1}{A}e^{A(4^{-1/K}-1)}<0=g(1),$$ it follows
that $\gamma$ is nonconstant. Since $\gamma $ is a subharmonic
increasing convex function of $\log r$, it follows that for $r>s$
$$g'(r+0)\ge g'(r-0)\ge g'(s+0)\ge g'(s-0)\ge 0.$$ Since it is non-constant, it satisfies in
particular that
$$g'(1-0)>0.$$ Notice that the last inequality is also a consequence of E.
Hopf boundary point lemma, see e.g. \cite{hopf}. Therefore
$$-\frac{1}{A}+\frac{1}{A}e^{A(|w(z)|-1)}\le -\frac{1}{A}+\frac{1}{A}e^{A(h(r)-1)},$$ i.e.
\begin{equation}\label{h}|w(z)|\le h(r), |z|=r, w\in \mathcal{QC}(\Bbb U, B,K),$$ where
$$h(r)<1\text{ and } h'(1-0)>0.\end{equation}
It follows that
$$\frac{1-|z|^2}{1-|w(z)|^2}\le \frac{1-|z|^2}{1-|h(|z|)|^2}\le C(K).$$
\end{proof}
\begin{remark}
The previous lemma, in particular relation \eqref{h}  can be
considered as a version of Schwarz lemma for the class
${\mathcal{QC}(\Bbb U, B,K)}$. Let
$$\mathcal{F}(\Bbb U, B)=\{w:\Bbb U\to \Bbb U: w(0)=0, |\Delta w|\le
B|\nabla w|^2\}.$$ Then for $B=0$ the class $\mathcal{F}(\Bbb U, B)$
coincides with the class of harmonic functions of the unit disk into
itself satisfying the normalization $w(0)=0$ .  Using Schwarz lemma
for harmonic functions, it can be shown that
\begin{equation}\label{coco}\frac{1-|z|^2}{1-|w(z)|^2}\le \frac{\pi}{2},\quad
w\in \mathcal{F}(\Bbb U, 0).\end{equation} It would be of interest
to verify if quasiconformality is important for ${\mathcal{QC}(\Bbb
U, B,K)}$. In other word do there hold \eqref{coco} for some
constant $C=C(B)$ instead of $\pi/2$ for the class $\mathcal{F}(\Bbb
U, B)$.
\end{remark}
\begin{lemma}\label{lica}
Let $(z_n)$ be an arbitrary sequence of complex numbers from the
unit disk. Assume that $w$ satisfies the conditions of
Theorem~\ref{tt}. Let  $p_n$ and $q_n$ be M\"obius transformations,
of the unit disk onto itself such that $p_n(w(z_n))=0$ and
$q_n(0)=z_n$. Take $w_n= p_n\circ w\circ q_n$. Then 
we have

a) $$|\nabla w_n|\le \frac{C_1(K)}{1-|z|},$$

b) $$|\Delta w_n|\le\frac{C_2(K)}{(1-|z|)^4},$$ and

c) $$|\Delta w_n|\le\frac{C_3(K)}{(1-|z|)^2}|\nabla {w_n}|^2.$$
\end{lemma}
\begin{proof}
Take
\begin{equation}\label{pn}p=p_n(w) = \frac{w-w(z_n)}{1-w\overline{w(z_n)}}\end{equation} and
\begin{equation}\label{qn}q=q_n(z) = \frac{z+z_n}{1+z\overline{z_n}}.\end{equation}
It is evident that $$w_n(z) = p_n\circ w\circ q_n$$ is a $K-$q.c.
mapping of the unit disk onto itself.

Next we have \begin{equation}\label{ee}(w_n)_ z = p_n'w_{q}
q_n'\text{  and  } (w_n)_{\bar z} = p_n'w_{\bar q}
\overline{q_n'}.\end{equation} Further we derive
\[\begin{split}(w_n)_{z\bar z} &= ({(p_n\circ w\circ q_n)}_z)_{\bar z} \\&= (p_n'w_{q}
q_n')_{\bar z} = p_n'' w_{\bar q} \overline{q_n'} w_{q} q_n' + p_n'
w_{q\bar q} \overline{q_n'} q_n'\\& = p_n''|q_n'|^2 w_{q} w_{\bar q}
+ p_n'|q_n'|^2w_{q\bar
q}
.\end{split}\] Thus \begin{equation}\label{firfir}(w_n)_{z\bar z} =
|q_n'|^2(p_n'' w_{q} w_{\bar q} + p_n'w_{q\bar q}).\end{equation}
Using now
$$2|w_{q} w_{\bar q}|\le |\nabla w|^2=\frac{|\nabla
w_n|^2}{|q'_n|^2|p_n'|^2}$$ and
$$|w_{q\bar q}| \le B |\nabla w|^2$$
we obtain
\begin{equation}\label{eqop}|(w_n)_{z\bar z}|\le\frac 12 \left(\frac{|p_n''|}{|p_n|^2} +
\frac{B}{|p_n'|}\right)|\nabla w_n|^2.\end{equation}
Now we have \begin{equation}\label{qq}|q_n'|
=\frac{1-|q_n(z)|^2}{1-|z|^2}=
\frac{1-|z_n|^2}{|1+z\overline{z_n}|^2}\le
\frac{1-|z_n|^2}{(1-|z|)^2},\end{equation}
\begin{equation}\label{pp}|p_n'|
=\frac{1-|p_n(w(q_n(z)))|^2}{1-|w(q_n(z))|^2}=
\frac{1-|w(z_n)|^2}{|1-w(q_n(z))\overline{w(z_n)}|^2}\le
\frac{2}{1-|w(z_n)|^2}\end{equation} and
\begin{equation}\label{rr}|p_n''| =
\frac{(1-|w(z_n)|^2)|w(z_n)|}{|1-w(q_n(z))\overline{w(z_n)}|^3}\le
\frac{8}{(1-|w(z_n)|^2)^2}.\end{equation} From \eqref{ee} --
\eqref{rr} and \eqref{ccoo} we obtain
\begin{equation}\label{mos}
|(w_n)_z|\le \frac{C(K)}{1-|z|}, \quad  |(w_n)_{\bar z}|\le
\frac{C(K)}{1-|z|}
\end{equation}
and
\begin{equation}\label{ww}|q_n'|^2 \left(|p_n''| + B|p_n'|\right)
\le
8\frac{(1-|z_n|^2)^2}{(1-|z|)^4}\left(\frac{1+B}{(1-|w(z_n)|^2)^2}\right).\end{equation}
From \eqref{mos} we infer the statement a).  Combining now
\eqref{ccoo}, \eqref{firfir} and \eqref{ww} we obtain
\begin{equation}\label{epara}
|(w_n)_{z\bar z}|\le \frac{8C(K)^2(1+B)}{(1-|z|)^4}.
\end{equation} Thus b) follows at once.
Let us now estimate the sequence $$S_n =\frac{|p_n''|}{{|p_n'|}^2} +
\frac{B}{|p_n'|}.$$
First of all
$$\frac{p_n''}{{p_n'}^2} =
\frac{2\overline{w(z_n)}(1-w_n(z)\overline{w(z_n)})}{1-|w(z_n)|^2}.$$
Hence
\[\begin{split}\left|\frac{p_n''}{{p_n'}^2}\right| &=
\frac{2\overline{|w(z_n)|}|1-w_n(z)\overline{w(z_n)}|}{1-|w(z_n)|^2}\\&\le
\frac{2|\overline{w(z_n)}||(w(\frac{z+z_n}{1+z\overline{z_n}})-w(z_n))\overline{w(z_n)}|}{1-|w(z_n)|^2}
+ 2. \end{split}\] To continue observe that
\begin{equation}\label{helpew}\begin{split}\left|w(\frac{z+z_n}{1+z\overline{z_n}})-w(z_n)\right|&\le |\nabla
w|_\infty \frac{|z|(1-|z_n|^2)}{|1+\overline{z_n} z|^2}\\&\le|\nabla
w|_\infty\frac{|z|(1-|z_n|^2)}{(1-|z|)^2}.\end{split}\end{equation}
Recall that by Proposition~\ref{arbitrarydomain} we have $|\nabla
w|_\infty<\infty$. Thus, by using \eqref{ccoo} we get
$$\left|\frac{p_n''}{{p_n'}^2}\right| \le 2+ \frac{|z||\nabla
w|_\infty}{(1-|z|)^2}\frac{1-|z_n|^2}{1-|w(z_n)|^2}\le 2 +
C(K)|\nabla w|_\infty \frac{2}{(1-|z|)^2},$$ i.e.
\begin{equation}\label{prima}
\left|\frac{p_n''(w(q_n(z)))}{{p_n'(w(q_n(z)))}^2}\right|\le  2 +
C(K)|\nabla w|_\infty \frac{2}{(1-|z|)^2}.
\end{equation}
Similarly, as $$\frac{1}{p_n'(w(q_n(z)))} =
\frac{(1-w(q_n(z))\overline{w(z_n)})^2}{1-|w(z_n)|^2}$$ we get,
according to \eqref{helpew} and \eqref{ccoo}, that
\begin{equation}\label{seconda}\left|\frac{1}{{p_n'(w(q_n(z)))}}\right| \le 1 + C(K)|\nabla w|_\infty
\frac{2}{(1-|z|)^2}.\end{equation} It follows that $$|S_n(z)| \le 1
+ \left(1+B\right)\left( 1 + C(K)|\nabla w|_\infty
\frac{2}{(1-|z|)^2}\right).$$  Combining with \eqref{eqop} we obtain
that the sequence $w_n$ satisfies the differential inequality
\begin{equation}\label{ise}
|\Delta w_n|\le\frac{1}{2} \left(1 + \left(1+B\right)( 1 +
C(K)|\nabla w|_\infty \frac{2}{(1-|z|)^2})\right)|\nabla w_n|^2.
\end{equation} This yields c) and the proof of Lemma~\ref{lica} is
completed.
\end{proof}
We will finish the proof of main result by using the following
lemma.
\begin{lemma}\label{loca}
Under the conditions of the previous lemma, there exists a
subsequence of $w_n$ converging to a mapping $w_0$ in the $C^1$ norm
uniformly on compact sets of the unit disk.
\end{lemma}
\begin{proof}[Proof of Lemma~\ref{loca}]
By \cite{FV} for example, a subsequence of $w_n$, also denoted by
$w_n$, converges uniformly to a $K$-quasiconformal mapping $w_0$ of
the closed unit disk onto itself. Let $0<r<1$ and take $\hat w_n(z)
= w_n(rz  )$, $z\in \Bbb U$. From \eqref{epara} it follows that $g_n
=\Delta \hat w_n$ is bounded. 
According to \eqref{e:POISSON1} it follows that
\begin{equation}\label{e:POISSON}\begin{aligned}\hat w_n(z)&=H_n(z)+G_n(z)=P[f_n](z)-
G[g_n](z)\\&: =\frac{1}{2\pi}\int_0^{2\pi}P(z,e^{i\varphi})w_n(re^{
i\varphi})d\varphi-
\int_{\mathbb{U}}G(z,\omega)g_n(\omega)\,dm(\omega),\,
\end{aligned}\end{equation} $|z|<1.$ Here $H_n$ is a harmonic function
with the same boundary data as $\hat w_n$ in $S^1$.
\\
We will prove that up to some subsequence, the sequence $\nabla \hat
w_n$ converges to $\nabla \hat w_0$ in $r\Bbb U$. As $\nabla \hat
w_n$ is uniformly bounded (see Lemma~\ref{lica}~a)), $|\nabla
G_n|\le \frac{2}{3}|g_n|$ (this inequality has been shown in
\cite{trans})) and $|g_n|\le M(K,r)$ (see Lemma~\ref{lica}~b)), it
follows that the family of harmonic maps $\nabla H_n$ is uniformly
bounded on $\Bbb U$. Therefore by Cauchy inequality we obtain
\begin{equation}\label{cauc}|\nabla^2 H_n|\le C\frac {|\nabla H_n|_\infty}{1-|z|}\le \frac{C_1}{1-|z|}. \end{equation}
\\
To continue observe that for $z\neq \omega$ we have
\[\begin{split}G_z (z,\omega)&=\frac{1}
{4\pi}\left(\frac{1}{\omega-z}-\frac{\bar \omega}{1-z\bar
\omega}\right)\\ &= \frac{1} {4\pi}\frac{(1-|\omega|^2)}{(
z-\omega)(z \bar \omega -1)},\end{split}\] and $$G_{\bar z}
(z,\omega)= \frac{1} {4\pi}\frac{(1-|\omega|^2)}{(\bar z-\bar
\omega)( \bar z \omega -1)}.$$
\\
Prove that the family of functions
\begin{equation}\label{F}F_n(z,z') =\partial G[g_n](z) -
\partial G[g_n](z') \text{ is uniformly continuous on
$\overline{\Bbb U} \times
\overline{\Bbb U}$.}\end{equation}  First of all $|g_n|_{\Bbb U}\le M(K,r)$.\\
Then
\[\begin{split}&|\partial G[g_n](z) - \partial G[g_n](z')|\\&\le \Phi(z,z'):=
\frac{M(K,r)}{4\pi}\int_{\mathbb{U}}\abs{{\frac{1-|\omega|^2}{(
z-\omega)(z \bar \omega -1)}}-{\frac{1-|\omega|^2}{( z'-\omega)(z'
\bar \omega -1)}}}\,dm(\omega). \end{split}\]
\\
We will prove that $\Phi(z,z')$ is continuous on $\overline{\Bbb U}
\times \overline{\Bbb U}$, and use the fact that $$\Phi(z,z)\equiv
0.$$
\\
In other world we will prove that
\begin{equation}\label{pril20}\lim_{n\to\infty}(z_n,z_n') =
(z,z')\Rightarrow\lim_{n\to \infty}\Phi(z_n,z_n') =
\Phi(z,z').\end{equation}
\\
In order to do so, we use the Vitali theorem (see \cite[Theorem
26.C]{halmos}):

\emph{Let $X$ be a measure space with finite measure $\mu,$ and let
$h_n: X\to \mathbb C$ be a sequence of functions that is uniformly
integrable, i.e. such that for every $\varepsilon>0$ there exists
$\delta>0,$ independent of $n,$ satisfying $$ \mu(E) <\delta
\implies \int_E |h_n|\, d\mu <\varepsilon. \eqno(\sharp)$$ Now: if\/
$\lim_{n\to\infty}h_n(x)=h(x)$ a.e., then
$$\lim_{n\to\infty}\int_X h_n\, d\mu=\int_X h\,d\mu.\eqno(\ddag)$$
In particular, if $$\sup_n\int_X |h_n|^p\,d\mu<\infty,\quad
\text{for some $p>1$},$$ then  $(\sharp)$ and $(\ddag)$ hold. }
\\
We will use the Vitali theorem for $$h_n(\omega) =
\left|{\frac{1-|\omega|^2}{( z_n-\omega)(z_n \bar \omega
-1)}}-{\frac{1-|\omega|^2}{( z_n'-\omega)(z_n' \bar \omega
-1)}}\right|,$$ defined in the unit disk.
\\
To prove \eqref{pril20}, it suffices to prove that $$M_p:=
\sup_{z,z'\in \mathbb U}
\int_{\mathbb{U}}\left(\frac{1-|\omega|^2}{|z - \omega|\cdot|1-\bar
z\omega|} +\frac{1-|\omega|^2}{|z' - \omega|\cdot|1-\bar z'\omega|}
\right)^p\,dm(\omega) \ <\infty,
$$ for $p=3/2$.
\\
Let $$I_{p}(z): =\int_{\mathbb{U}}\left(\frac{1-|\omega|^2}{|z -
\omega|\cdot|1-\bar z\omega|} \right)^p\,dm(\omega).$$ For a fixed
$z$, we introduce the change of variables
\[\begin{aligned} \frac{z-\omega}{1-\bar z\omega}=\xi,  \end{aligned} \]
or, what is the same \[\begin{aligned} \omega=\frac{z-\xi}{1-\bar
z\xi}.  \end{aligned} \] Therefore $$\begin{aligned}I_{p}(z)&=
\int_{\mathbb{U}}\left(\frac{1-|\omega|^2}{|z - \omega|\cdot|1-\bar
z\omega|}
\right)^p\,dm(\omega)\\[2ex]&
=
\int_{\mathbb{U}}\frac{(1-|z|^2)^{2-p}(1-|\omega|^2)^{p}}{|\xi|^p\,|1-\bar
z\xi|^4}\,dm(\xi)\\[2ex]& \le
(1-|z|^2)^{1/2}\int_0^1\rho^{-1/2}(1-\rho^2)^{3/2}\,d\rho\int_0^{2\pi}
|1-\bar z\rho e^{i\varphi}|^{-4}\,d\varphi\\[2ex]& \le
(1-|z|^2)^{1/2}\int_0^1
\rho^{-1/2}(1-\rho^2)^{3/2}(1-|z|\rho)^{-3}\,d\rho.
\end{aligned}$$
From the elementary inequality
$$ \int_0^1 \rho^{-1/2}(1-\rho^2)^{3/2}(1-|z|\rho)^{-3}\,d\rho\, \le
\, C(1-|z|^2)^{-1/2},$$ it follows that $$\sup_{z\in \Bbb
U}I_{p}(z)<\infty.$$ Finally,  Holder inequality implies $$M_p\le
2^{p-1}\sup_{z,z'\in \Bbb U}(I_{p}(z) + I_{p}(z'))<\infty.$$ This
means that $\Phi$ is uniformly continuous on $\overline{\Bbb U}
\times \overline{\Bbb U}.$ Using the fact that $\Phi(z,z)\equiv 0$,
it follows that for $\varepsilon
>0$ there exists $\delta >0$ such that  $$|z-z'|\le \delta \Rightarrow|\partial G[g_n](z) - \partial G[g_n](z')|\le\Phi(z,z')\le
\varepsilon.  $$ Similarly we obtain that the family
\begin{equation}\label{F1}\overline F_n(z,z') =\bar\partial G[g_n](z) -
\bar \partial G[g_n](z') \text{ is uniformly continuous on
$\overline{\Bbb U} \times \overline{\Bbb U}$.}\end{equation} By
\eqref{cauc}, \eqref{F} and \eqref{F1} and Arzela--Ascoli theorem,
there exists a subsequence of $w_n$  which will be also denoted by
$w_n$ converging to $w_0$ in $C^1$ metric uniformly on the disk
$r^2\overline{\Bbb U} = \{z: |z|\le r^2\}$:
$$\lim_{n\to \infty} w_n(z) = w_0(z)\text{ and }\lim_{n\to \infty} \nabla
w_n(z) = \nabla w_0(z)\ \ \  z\in r^2 \overline{\Bbb  U}.$$ Using
the diagonalisation procedure it follows the desired conclusion.
\end{proof}
\begin{proof}[Proof of Theorem~\ref{tt}] If $w(a)=0$ such that $a\neq 0$, then take the mapping
$w(\frac{z+a}{1+z\bar a})$ and also denote by $w$. It is obvious
that is satisfies the conditions of our theorem.

Assume that there exists a sequence of points $z_n$ such that
$\lim_{n\to\infty}\nabla w(z_n) = 0$. From Lemma~\ref{lica}, if
$|z|\le r^2<1$, it follows that there exists a constant $C^1_K(r)$
such that
\begin{equation}\label{detroit}|\Delta w_n|\le {C^1_K}(r)|\nabla
w_n|.\end{equation} Let $u_n + iv_n = w_n$. Let $A$ and $B$ be
defined by $$A:=|\nabla u_n|^2=2(|{u_n}_z|^2+|{u_n}_{\bar
z}|^2)=\frac 12(|w_z+\overline{{w_n}_{\bar z}} |^2+|{w_n}_{\bar
z}+\overline{{w_n}_z}|^2)$$ and $$B:=|\nabla
v_n|^2=2(|{v_n}_z|^2+|{v_n}_{\bar z}|^2)=\frac
12(|{w_n}_z-\overline{{w_n}_{\bar z}} |^2+|{w_n}_{\bar
z}-\overline{{w_n}_z}|^2).$$ Then
$$\frac{A}{B}=\frac{|1+\mu|^2}{|1-\mu|^2}$$ where
$$\mu=\frac{\overline{{w_n}_{\bar z}}}{{{w_n}_z}}.$$ Since $|\mu|\le k$ it follows that
\begin{equation}\label{help} \frac{(1-k)^2}{(1+k)^2}\le \frac AB
\le \frac{(1+k)^2}{(1-k)^2}.
\end{equation}
From \eqref{detroit} and \eqref{help} it follows that there exists a
constant $C^2_K(r)$ such that
\begin{equation}\label{korrik}|\Delta u_n|\le {C^2_K}(r)|\nabla u_n|.\end{equation}
From Lemma~\ref{loca} $$\lim_{n\to \infty}||\nabla w_n - \nabla
w_0||_{r^2\Bbb U}+ ||w_n - w_0||_{r^2\Bbb U} = 0.$$
We next have
$$\nabla w_n(0)=\frac{1-|z_n|^2}{1-|w(z_n)|^2}|\nabla w(z_n)|.$$
According to \eqref{ccoo}
$$\frac{1-|z_n|^2}{1-|w(z_n)|^2}\le C(K).$$
It follows that $\nabla w_0(0) = 0,$ and consequently
\begin{equation}\label{ek}
\nabla u_0(0) = 0.
\end{equation}
By Remark~\ref{nota} $w$ is a q.c. diffeomorphism. It follows that
$w_n$, $n\ge 1$ are quasiconformal diffeomorphisms. From
\eqref{help} we obtain that $\nabla u_n\neq 0$. Thus all the
conditions of Proposition~\ref{propi} are satisfied with $D = \{z:
|z|\le r^2\}$ and $u_n=\mathrm{Re}\ w_n$. This infers that
$u_0\equiv 0$ which is a contradiction, because $w_0$ is a
quasi-conformal mapping.

The rest of the proof follows from the fact that $|\nabla w|$ is
bounded below and above by positive constants and quasiconformality.
\end{proof}
\section{Applications}
The mapping $w_0$ produced in Lemma~\ref{loca} exists without the a
priori assumption that $\nabla w(z_n)\rightarrow 0$. In the
following proof we prove that $w_0$ is a harmonic mapping with
respect to an appropriate conformal metric $\rho_0$ depending on the
initial metric $\rho$ and on the sequence $z_n$.
\begin{proof}[{ Proof of Theorem~\ref{coro}}]
Firs of all, using \eqref{ee} we have
$${w_n}_z{\overline{w_n}}_{ z} = |p_n'|^2 w_{q_n}\overline{w}_{q_n}  q_n'(z)^2.$$
On the other hand, since $w$ is $\rho$ harmonic it follows that
$$\Psi_w(q_n(z))=\rho^2(w(q_n(z)))w_{q_n}\overline{w}_{q_n} q_n'(z)^2$$ is
analytic. Thus $w_n$ is $\rho_n$ harmonic for
$$\rho_n^2(w_n(z)) = \frac{\rho^2(w(q_n(z)))}{|p_n'(w(q_n(z)))|^2(1-|z_n|^2)^2}.$$
This means that the Hopf differential $$\Psi_n(z) =
\rho_n^2(w_n(z)){w_n}_z{\overline{w_n}}_{ z}
$$ of $w_n$ is analytic. According to Proposition~\ref{arbitrarydomain} and \eqref{qq} it follows that \begin{equation}\label{montel}|\Psi_n(z)|\le
\frac{C}{(1-|z|)^4}.\end{equation} Therefore by Montel's theorem, up
to some subsequence $\Psi_n$ converges to some analytic function
$\Psi_0$ on the unit disk. On the other hand, up to some subsequence
(according to Lemma~\ref{loca})
$${w_n}_z{\overline{w_n}}_{ z}$$ converges uniformly in compact sets of the unit disk to  $${w_0}_z\overline{w_0}_z.$$
Also we have
\[\begin{split}&\rho_n(w_n(z))=\frac{\rho(w(q_n(z)))}{|p_n'(w(q_n(z)))|(1-|z_n|^2)}\\&\le
C\frac{|1-w(q_n(z))\overline{w(z_n)}|^2}{(1-|w(z_n)|^2)(1-|z_n|^2)}\\&\le
C\frac{(1-|w(z_n)|^2)^2+2|w(q_n(z))-w(z_n)|(1-|w(z_n)|^2)
+|w(q_n(z))-w(z_n)|^2}{(1-|w(z_n)|^2)(1-|z_n|^2)}\\&\le
C\left(\frac{1-|w(z_n)|^2}{1-|z_n|^2}+2\frac{|w(q_n(z))-w(z_n)|}{1-|z_n|^2}+\frac{|w(q_n(z))-w(z_n)|^2}{(1-|w(z_n)|^2)(1-|z_n|^2)}\right).\end{split}\]
To continue use again the fact that $|\nabla w|_\infty<\infty$.
Therefore
$$|w(q_n(z))-w(z_n)| \le |\nabla w|_\infty |q_n(z)-q_n(0)| = |\nabla
w|_\infty \frac{|z|(1-|z_n|^2)}{|1+z\overline {z_n}|}$$ and hence
$$\frac{|w(q_n(z))-w(z_n)|}{1-|z_n|^2}\le \frac{|\nabla
 w|_\infty}{1-|z|}.$$ Combining the previous inequalities and \eqref{ccoo} we
 obtain $$\rho_n(w_n(z)) \le C\left(2|\nabla
 w|_\infty +\frac{2|\nabla
 w|_\infty}{1-|z|}+\frac{C(K)|\nabla
 w|_\infty^2}{(1-|z|)^2}\right).$$ It follows that $$\rho_n^2(w_n(z))\to
B(z):=\frac{\Psi_0(z)}{{w_0}_z\overline {w_0}_z},$$ where the
quantity
$$B(z)=\frac{\Psi_0(z)}{{w_0}_z\overline {w_0}_z} $$
is finite for $z\in \Bbb U$. \\Thus
\begin{equation}\label{efe}\rho^2_n(t)\to\rho_0^2(t) :=
B(w_0^{-1}(t)).\end{equation} Without loss of generality assume that
$z_n\to 1$. $\rho_0$  is not identical to zero because, according to
\eqref{ccoo} and \eqref{double}\[\begin{split}\rho_0(0) &=
\lim_{n\to\infty}\frac{\rho(w(q_n(0)))}{|p_n'(w(q_n(0)))|(1-|z_n|^2)}\\&=\frac{\rho(w(1))}{\lim_{n\to
\infty}\frac{1-|z_n|^2}{1-|w(z_n)|^2}}
\ge\frac{\rho(w(1))}{C(K)}>0.\end{split}\] This means that $\rho_0$
is a metric on the unit disk.

We obtain that $w_0$ is a harmonic quasiconformal mapping of the
unit disk with respect to the metric $\rho_0$ defined in
\eqref{efe}.
\end{proof}
The next theorem implies Theorem~\ref{tito}.
\begin{theorem}
 Let $(\Sigma_1,\sigma)$ and $(\Sigma_2,\rho)$ be $C^{2,\alpha}$ surfaces,
 with $C^{2,\alpha}$ compact boundaries and of equal connectivities, such that $\sigma$ and $\rho$ are approximate analytic
metrics. Let $w: \Sigma_1 \to \Sigma_2$ be a harmonic homeomorphism.
Then the following conditions are equivalent:

a) $w$ is quasiconformal;

b) $w$ is bi-Lipschitz with respect to $\sigma$ and $\rho$;
\end{theorem}
\begin{remark}
Let us consider the case $\Sigma_1=\Sigma_2=\mathbb U$.  By
\eqref{double}, for $\varrho_1=\rho$ and $\varrho_2=\sigma$ there
exists a constant $P_k>0$ such that
$$P_k^{-1}|w_1-w_2|\le d_{\varrho_k}(w_1,w_2)\le P_k|w_1-w_2|, \text{$w_1,w_2\in
\Sigma_k$},\ \ k=1,2 .$$ Thus internal distance $d_{\varrho_k}$,
which is induced by the metric $\varrho_k$ in $\Sigma_k$ and
Euclidean metric are bi-Lipschitz equivalent.
\end{remark}

\begin{proof}
The proof depends on the following Korn-Lichtenstein-Kellogg's type
proposition of Jost.
\begin{proposition}\cite[Theorem~3.1]{jost}\label{thei} Suppose $S$ is a
surface with boundary, homeomorphic to a plane domain $G$ bounded by
$k$ circles via a chart $\psi: \overline G\mapsto S$. Suppose the
coefficients of the metric tensor of $S$ can be defined in this
chart by bounded measurable functions $g_{ij}$  with $g_{11} g_{22}
- g_{12}^2\ge \lambda >0$ in $G$. Then $S$ admits a conformal
representation $\tau\in H_1^2\cap C^{\alpha}(\bar B, \bar G)$, where
$B$ is a plane domain bounded by $k$ circles and $\tau$ satisfies
almost everywhere the conformality relations $$|\tau_x|^2 =
|\tau_y|^2, \ \text{and} \left <\tau_x,\tau_y \right> = 0$$ (Here
$(x,y)$ denote the coordinates of points in $B$, and norms and
products are taken with respect to the metric of $S$).

Furthermore, concerning higher regularity, $\tau$ is as regular as
$S$ , i.e. if $S$ is of class $C^{m,\alpha}(\bar B)$ ($m\in \Bbb N$,
$0<\alpha<1$) or in $C^\infty$ then also $\tau \in C^{m,\alpha}(\bar
B)$ or $\tau \in C^{\infty}(\bar B)$, respectively. In particular,
if $S$ is at least $C^{1,\alpha}$ then the conformality relations
are satisfied everywhere, and $\tau$ is a diffeomorphism.
\end{proposition}

We consider four cases.

(i)  $\Sigma_1$ and $\Sigma_2$ are compact surfaces without
boundary.  The theorem is well-known, since every harmonic
homeomorphism is a diffeomorphism and consequently it is
bi-Lipschitz.

(ii) $\Sigma_1$ and $\Sigma_2$ are conformally equivalent to the
unit disk. Then for $i =1,2$ there exists a conformal mapping
$\tau_i: \Bbb U\to \Sigma_i$. Let $w$ be $K$-quasiconformal. Take
$\hat w=\tau_2^{-1}\circ w\circ\tau_1$. Let us show that $\hat w$ is
a harmonic
 mapping of the unit disk onto itself with an approximate
analytic metric. First of all
$$\tau_2' \hat w _z = \partial w \tau_1',$$
$$\tau_2' \hat w _{\bar z} = \bar\partial w \overline{\tau_1'},$$
and
$$\tau_2'' \hat w_z\cdot \hat w_{\bar z} + \tau_2' \hat w_{z\bar z} = \partial\bar\partial w |\tau_1'|^2.$$ Thus

\begin{equation}\label{s2.tran1}
\frac{\hat w_{z\bar z}}{\hat w_z\cdot \hat w_{\bar z}}
=-\frac{\tau_2''}{(h_2')^2}+\tau_2' \frac{\partial\bar\partial
w}{\partial w\cdot\bar\partial
w}=-\frac{\tau_2''}{(\tau_2')^2}-2\tau_2'
\frac{\partial\sigma_2}{\sigma_2}.
\end{equation}
By using \eqref{double}, it follows that the coefficients of the
metric tensor

$$g_{ij} = \left\{
                  \begin{array}{ll}
                  \rho(z), & \hbox{if $i=j$;} \\
                  0, & \hbox{if $i\neq j$}
                 \end{array}
                        \right.,
              $$
satisfy the condition of Proposition~\ref{thei}. Therefore
$|\tau_2''|$, $|\tau_2'|$ and $\frac{1}{|\tau_2'|^2}$ are bounded.
Thus $\hat w$ is a quasiconformal harmonic mapping with respect to
an approximate analytic metric. Theorem~\ref{tt} implies that $\hat
w$ is bi-Lipschitz with respect to Euclidean metric. Since $\tau_1$
and $\tau_2$ are diffeomorphisms up to the boundaries, the mapping
is bi-Lipschitz as well. It follows that $w$ is bi-Lipschitz (with
respect to internal metrics). Therefore  $a) \Rightarrow b)$.

Since every Euclidean bi-Lipschitz is quasiconformal, according to
the previous facts we obtain $b) \Rightarrow a)$.

(iii) $\Sigma_1$ and $\Sigma_2$ are homeomorphic to a plane domain
$G$ bounded by $k$ circles. Let $\tau_1\colon B \to \Sigma_1$ and
$\tau_2\colon B\to \Sigma_2$ be conformal mappings produced in
Proposition~\ref{thei}, where $B$ is a plane domain bounded by $k$
circles. Take $ \hat w=\tau_2^{-1}\circ w\circ\tau_1$.

For every boundary point $t\in\partial \Sigma_1$, there exists a
neighborhood $B(t)\subset D$ of $s=\tau_1^{-1}(t)$ (with respect to
the boundary of $D$), which is conformally equivalent to the unit
disk. Let $\tau_3\colon \Bbb U \to B(t)$ and $\tau_4\colon \Bbb U\to
\hat w(B(t))$ be Riemann conformal mappings. Take now $\hat w_t =
\tau_4^{-1}\circ\hat w\circ \tau_3$. According to the case (ii),
there exists  a positive constant $C_t(K)$ such that:
$$1/C_t(K)\le  |\nabla \hat w_t(z)|\le C_t(K), z\in \Bbb U.$$
Using the Schwarz's reflexion principle to the mappings $\tau_3$ and
$\tau_4$  it follows that there exists a positive constant $C'_t(K)$
such that $$1/C'_t(K)\le |\nabla \hat w(z)|\le C'_t(K), z\in
B_t(K),$$ where $B_t(K)\subset B$ is a neighborhood of $s$. Since
$\tau_1$ and $\tau_2$ are diffeomorphisms, it follows that there
exists a constant $C''_t(K)$ such that $$1/C''_t(K)\le |\nabla
w(z)|\le C''_t(K), z\in \Sigma_t(K),$$ where $\Sigma_t(K)\subset
\Sigma_1$ is a neighborhood of $t$. Since $\partial \Sigma_1$ is
compact, there exists a positive constant $C'(K)$ such that
$$1/C'(K)\le |\nabla  w(z)|\le C'(K), z\in \Sigma(K),$$ where
$\Sigma(K)$ is a neighborhood of $\partial \Sigma_1$. Finally we
conclude that there exists a positive constant $C''(K)$ such that
$$1/C''(K)\le |\nabla  w(z)|\le C''(K), z\in \Sigma_1.$$  The
conclusion follows from the relations $$|\nabla
w^{-1}(w(t))|=\frac{1}{l(\nabla w(t))},$$
$$|\nabla w|\le K l(\nabla(w)),$$ the main value theorem and the
fact that the surfaces are quasi-convex.

(iv) The general case. Let $\gamma$ be one of the boundary
components of $\Sigma_1$. Then $\delta=w(\gamma)$ is a boundary
component of $\Sigma_2$. Assume that $\gamma'\subset
\Sigma_1\setminus \gamma$ is a $C^{2,\alpha}$ Jordan curve homotopic
to $\gamma$. Then $\delta'=w(\gamma')\in C^{2,\alpha}$ is homotopic
to $\delta$. Let $A\subset \Sigma_1$ be the annulus generated by
$\gamma$ and $\gamma'$. Applying the case (iii) to the mapping
$w:A\to w(A)$ we obtain the desired conclusion.

\end{proof}

\begin{remark}

Let $\lambda$ be the hyperbolic metric defined in \eqref{hiperbola}.
In \cite[Theorem~13]{wan} is proved that a $\lambda$ harmonic
self-mapping of the unit disk is q.c. if and only if the function
$$\Psi= \frac{(1-|z|^2)^2w_z\overline{w_{\bar z}}}{(1-|w(z)|^2)^2}$$
is bounded.  Moreover, concerning the hyperbolic metric, Wan showed
that if $w$ is $k$-q.c. $\lambda$ harmonic, then it is a hyperbolic
bi-Lipschitz self-mapping of the unit disk. See also \cite{MMM}.

The previous method gives a short proof of the theorem, that a q.c.
harmonic mapping of the hyperbolic disk onto itself is bi-Lipschitz
(one direction of Wan's theorem).

To do so, denote by $e(w)$ the hyperbolic energy of a q.c. harmonic
mapping of the unit disk onto itself:

$$e(w) = \frac{(1-|z|^2)^2}{(1-|w(z)|^2)^2}(|w_z|^2+|w_{\bar z}|^2).$$

Assume there exists a sequence $(z_n)$ such that $e(w)(z_n)\to
\infty$, or $e(w)(z_n) \to 0$, as $n\to \infty$.  Take $w_n =
p_n(w(q_n(z)))$, where $p_n$ and $q_n$ are M\"obius transformations
of the unit disk onto itself satisfying the conditions $p_n(w(z_n))
= 0$ and $q_n(0)=z_n$. Then, $w_n(0)=0$ and up to some subsequence,
$w_n \to w_0$ where $w_0$ is quasiconformal and harmonic. By
\cite{sy} $\nabla w_0(0)\neq 0$.

But here we have \[\begin{split}2|\nabla w_0(0)|_2^2 &=\lim_{n\to
\infty}2|\nabla w_n(0)|_2^2
\\&=\frac{(1-|z_n|^2)^2}{(1-|w(z_n)|^2)^2}(|w_z(z_n)|^2+|w_{\bar
z}(z_n)|^2) \\&= \lim_{n\to \infty} e(w)(z_n)=\infty\ \ \text {or}\
\ = 0.\end{split}\]  This is a contradiction. Therefore there exists
a constant $C\ge 1$ such that
$$\frac 1C\lambda(z)|dz|\le w^*(\lambda (z) |dz|)\le C \lambda(z) |dz|$$ as
desired.

\end{remark}

\subsection{Open problems} a)
It is not known by the author if the mapping $w_0$ produced in
Corollary~\ref{coro} is bi-Lipschitz with respect to the Euclidean
metric or what is a bit more, whether $\rho_0$ is an approximate
analytic metric. This problem is open even for $\rho$ being a
Euclidean metric. Also it is an interesting question if two
sequences $z_n$ and $z_n'$ converges to the same boundary point
$z_0$, do they induce the same harmonic mapping $w_0$ or at least
the same metric $\rho_0$.

b) The Gauss curvature of a metric $\rho$ is given by $$K =
-\frac{\Delta \log \rho}{\rho^2}.$$ Thus the Gauss curvature is
positive if and only if \begin{equation}\label{hhee}\Delta \rho(z)
\le \frac{1}{\rho(z)}|\nabla \rho(z)|^2,\ \ z\in\Bbb
U.\end{equation} Heinz-Bersnetin theorem (\cite{EH}) states that: if
\begin{equation}\label{he}|\Delta \rho(z)| \le \frac{1}{\rho(z)}|\nabla \rho(z)|^2,\ \ z\in\Bbb U\end{equation}
and $\rho\in C^{1,\alpha}(S^1)$ then $|\nabla \rho|$ is bounded,
provided that $\rho(z)$ is bounded below away from zero. Therefore
$\rho$ is an approximate analytic metric.

Under these conditions on the metric $\rho$ the main theorem is
true. The question arises whether the condition \eqref{he} can be
replaced by \eqref{hhee}.

\subsection*{Acknowledgment} I thank the referee for pointing out an error
in the previous version of the paper.
%
\bibliographystyle{amsplain}
\providecommand{\bysame}{\leavevmode\hbox
to3em{\hrulefill}\thinspace}
\providecommand{\MR}{\relax\ifhmode\unskip\space\fi MR }
\providecommand{\MRhref}[2]{%
  \href{http://www.ams.org/mathscinet-getitem?mr=#1}{#2}
} \providecommand{\href}[2]{#2}

\end{document}